\documentclass[a4paper,12pt]{amsart}
\usepackage{times} 
\usepackage{latexsym,amscd,amssymb,url}
\usepackage[latin1]{inputenc} 	
\usepackage[T1]{fontenc}         
\usepackage{ae}									
\usepackage[all,cmtip]{xy}  
\usepackage{graphicx}
\usepackage{color}
\usepackage{hyperref}
\usepackage{enumerate}

\newtheorem{theorem}{Theorem}

\theoremstyle{plain}

\newtheorem{question}[theorem]{Question}

\newtheorem{definition}[theorem]{Definition}

\newtheorem{lemma}[theorem]{Lemma}

\newtheorem{proposition}[theorem]{Proposition}

\newcommand{\Z}{\mathbb Z}
\newcommand{\Q}{\mathbb Q}
\newcommand{\PP}{\mathbb P}

\newcommand{\C}{\mathbb C}

\newcommand{\CP}{\mathbb P}

\newcommand{\Pic}{\operatorname{Pic}}

\newcommand{\Sing}{\operatorname{Sing}}

\newcommand{\dashedlongrightarrow}{\xymatrix@1@=15pt{\ar@{-->}[r]&}}
\renewcommand{\longrightarrow}{\xymatrix@1@=15pt{\ar[r]&}}
\renewcommand{\mapsto}{\xymatrix@1@=15pt{\ar@{|->}[r]&}}
\renewcommand{\twoheadrightarrow}{\xymatrix@1@=15pt{\ar@{->>}[r]&}}
\newcommand{\hooklongrightarrow}{\xymatrix@1@=15pt{\ar@{^(->}[r]&}}
\newcommand{\congpf}{\xymatrix@1@=15pt{\ar[r]^-\sim&}}
\renewcommand{\cong}{\simeq}

\begin{document}
\title[Chern numbers of uniruled threefolds]{Chern numbers of uniruled threefolds} 

\author{Stefan Schreieder}
\address{Mathematisches Institut, Ludwig--Maximilians--Universit\"at M\"unchen, Theresienstr.\ 39, 80333 M\"unchen, Germany} 
\email{schreieder@math.lmu.de}

\author{Luca Tasin} 
\address{Dipartimento di Matematica F.\ Enriques, Universit\`a degli Studi di Milano, Via Cesare Saldini 50, 20133 Milano, Italy} 
\email{luca.tasin@unimi.it}

\date{\today}


\begin{abstract}  
In this paper we show that the Chern numbers of a smooth Mori fibre space in dimension three are bounded in terms of the underlying topological manifold.
We also generalise a theorem of Cascini and the second named author on the boundedness of Chern numbers of certain threefolds to the case of negative Kodaira dimension.
\end{abstract}

\maketitle

\section{Introduction}

One of the most basic numerical invariants of a compact complex manifold are its Chern numbers.
While these numbers depend only on the topological type of the complex structure of the tangent bundle, they are in general not invariants of the underlying topological manifold, but really depend on the complex structure. 
In fact, answering a question of Hirzebruch from 1954, all linear combinations of Chern and Hodge numbers which are topological invariants of smooth complex projective varieties have recently been determined in \cite{kotschick-PNAS,kotschick-advances,kotschick-schreieder}.

Generalising Hirzebruch's question, Kotschick asked \cite{kotschick-JTOP} (see also \cite{LSZ}) whether the topology of the underlying smooth manifold determines the Chern numbers of smooth complex projective varieties at least up to finite ambiguity.
In \cite{schreieder-tasin}, we have shown that in dimension at least four, this question has in general a negative answer.
That is, there are smooth real manifolds that carry infinitely many complex algebraic structures such that the corresponding Chern numbers are unbounded, except for $c_n$, $c_1c_{n-1}$ and $c_2^2$ which are known to be bounded (see \cite{LW90} for the non-trivial one  $c_1c_{n-1}$).
This result left however open the case of threefolds, where it remains unknown whether $c_1^3$ is determined up to finite ambiguity by the underlying smooth manifold.

In \cite{cascini-tasin}, Cascini and the second named author started to investigate the boundedness question for Chern numbers via methods from the minimal model program, see also \cite{MST,ST} for further developments.
In dimension three, the approach in \cite{cascini-tasin} is motivated by the Miyaoka--Yau inequality, which implies that for a minimal smooth complex projective threefold of non-negative Kodaira dimension, $c_1^3$ can be bounded in terms of the Betti numbers of $X$, see e.g.\ \cite[Proposition 9]{ST}.
This observation makes it natural to approach the boundedness of $c_1^3(X)$ by trying to bound the effect on $c_1^3(X)$ of the steps in the minimal model program for $X$.
This leads to a positive answer for the boundedness question for many smooth projective threefolds of non-negative Kodaira dimension whose minimal model program is a composition of 
blow-downs to points and smooth curves in smooth loci, see \cite[Corollary 1.5]{cascini-tasin}.

In this paper we focus on the case of threefolds of negative Kodaira dimension.
The main difficulty that we face in this case is that the aforementioned Miyaoka--Yau inequality, which was essential for the case of non-negative Kodaira dimension, does not hold any longer.
It is also known by examples of LeBrun \cite{lebrun}, that the boundedness does not hold in the non-K\"ahler case.
Nonetheless, for any smooth K\"ahler threefold $X$ we can run a minimal model program thanks to \cite{HP15,HP16}.
If $X$ is uniruled then we arrive at a Mori fibre space $Y\to B$, i.e.\ a K\"ahler threefold $Y$ with at most terminal singularities together with a morphism of relative Picard rank one with connected fibres to a complex K\"ahler variety $B$ of smaller dimension whose general fibre is Fano.

The first result of this paper is the following.

\begin{theorem} \label{thm:MFS}
	Let $(X_i)_{i\geq 0}$ be a sequence of Mori fibre spaces, where $X_i$ are smooth K\"ahler threefolds.
	If each $X_i$ is homeomorphic to $X_0$, then the sequence of Chern numbers $c_1^3(X_i)$ is bounded.
\end{theorem}

The above result should be compared to the fact that all known examples of sequences of homeomorphic varieties with unbounded Chern numbers are Mori fibre spaces, and in fact projective bundles (see \cite{schreieder-tasin}).
We therefore believe that together with the aforementioned results from \cite{cascini-tasin}, the above theorem puts forward strong evidence for the conjecture that the Chern numbers of smooth projective threefolds are determined up to finite ambiguity by the underlying smooth manifold.

If $X \to B$ is a Mori fibre space and $X$ is a smooth K\"ahler threefold, then there are three main cases to consider, depending on the dimension of $B$. 
If $B$ is a point, then $X$ is a Fano variety and we conclude because Fano varieties of fixed dimension form a bounded family. 
If $B$ is a curve, then it is smooth projective and $X$ is also projective. 
Since the Pontryagin classes are up to torsion homeomorphism invariants by Novikov's theorem \cite{novikov}, \cite[Proposition 26]{ST} proves the above theorem in case all but finitely many of the $X_i$ are Mori fibre spaces over points or curves.
Using Novikov's theorem \cite{novikov} once again, Theorem \ref{thm:MFS} thus follows from the following more precise result about Mori fibre spaces over surfaces, where we denote by $H^\ast_{tf}(X,\Z)$ the quotient of $H^\ast(X,\Z)$ by the subgroup of all torsion classes.

\begin{theorem} \label{thm:MFS:surfaces}
	Let $(X_i)_{i\geq 0}$ be a sequence of smooth K\"ahler threefolds admitting a conic bundle structure $f_i:X_i\longrightarrow S_i$ of relative Picard number 1 over a smooth K\"ahler surface $S_i$.
	If there is an isomorphism of graded rings $H^\ast_{tf}(X_i,\Z)\cong H^\ast_{tf}(X_0,\Z)$ which respects the first Pontryagin classes, then the sequence of Chern numbers $c_1^3(X_i)$ is bounded. 
\end{theorem}
 
In view of Theorem \ref{thm:MFS}, it is therefore natural to wonder if we can also bound the Chern numbers of certain threefolds of negative Kodaira dimension which are not necessarily Mori fibre spaces themselves. 
Our next result achieves this by generalising   
\cite[Corollary 1.5]{cascini-tasin} to the case of negative Kodaira dimension. 
To state it, recall that for any smooth complex projective threefold $X$, there is a cubic form $F_X$ on $H^2(X,\Q)$, given by cup product.
For technical reasons, we will assume that the discriminant $\Delta_{F_X}$ of the cubic form is non-zero.

\begin{theorem}\label{thm:uniruled}
	Let $X$ be a smooth complex projective threefold which is uniruled and let $F_X$ be its associated cubic. Assume that $\Delta_{F_X} \ne 0$ and that there exists a birational morphism $f\colon X\to Y$ onto a Mori fibre space $Y$, which is  obtained as a composition of  divisorial contractions to points and blow-downs to smooth curves in smooth loci. 
	
	Then there exists a constant $D$ depending only on the topology of the $6$-manifold underlying $X$ such that
	$$
	|K_X^3| \le D.
	$$
\end{theorem}

A major step in proving Theorem \ref{thm:uniruled} is Proposition \ref{prop:triple} (cf.\ \cite[Theorem 1.3(2)]{cascini-tasin}), where we show that in the assumptions of Theorem \ref{thm:uniruled}, most of the topological invariants of $Y$ are determined (up to finite ambiguity) a priori by the smooth manifold underlying $X$. It would be interesting to understand to what extend this is true in general (see \cite{Chen} for the case of Betti numbers):

\begin{question}\label{q:topMMP}
Let $X$ be a smooth complex projective threefold with cubic form $F_X$ and first Pontryagin class $p_1(X)$.
Let $\mathcal P$ be the set of pairs $(F_Y,p_1(Y))$, taken up to isomorphism, such that there exists an MMP $X \dashrightarrow Y$. 
Is the set $\mathcal P$  determined by the pair of invariants $(F_X,p_1(X))$ of $X$ up to finite ambiguity?
\end{question}

\subsection{Conventions}  
All manifolds are closed and connected.
A K\"ahler manifold is a complex manifold which admits a K\"ahler metric. 
For any (K\"ahler) manifold $X$, we denote by $H^\ast_{tf}(X,\Z)$ the quotient $H^\ast(X,\Z)/H^\ast(X,\Z)_{tors}$, where $ H^\ast(X,\Z)_{tors}$ denotes the torsion subgroup of $H^\ast(X,\Z)$. 

\section{Mori fibre spaces over surfaces}

The starting point of our investigation is the following lemma.

\begin{lemma}\label{lemma:bettiMfs}(\cite[Sec. 7.1]{Fano})
	Let $f: X \to S$ be a  Mori fibre space such that $X$ is a smooth projective threefold and $S$ is a surface. 
	Then  
	\begin{enumerate}[(i)]
		\item $f: X \to S$ is a standard (i.e.\ relative Picard number 1) conic bundle and $S$ is smooth;
		\item the discriminant $D \subset S$  of $f$ is either empty or a reduced curve with at worst ordinary double points;
		\item $e(X)= 2(e(S)-p_a(D)+1)$, $b_1(X)=b_1(S)$ and $b_3(X)=2(b_1(X)+p_a(D)-1)$;
		\item $D \equiv -f_*K^2_{X/S}$ and $-4K_S \equiv f_*K_X^2 +D$. 
	\end{enumerate}
\end{lemma}

We will also use the following lemma.

\begin{lemma} \label{lem:f_*p1}
	Let $f: X \to S$ be a  Mori fibre space such that $X$ is a smooth projective threefold and $S$ is a surface. 
	Then,
	$$
	f_\ast p_1(X)\equiv -3D ,
	$$
	where $p_1(X)$ is the first Pontryagin class of $X$ and  $D\subset S$ denotes the discriminant curve of $f$.
\end{lemma}
\begin{proof}
	Since $S$ is smooth projective by Lemma \ref{lemma:bettiMfs}, the N\'eron--Severi group of $S$ is generated by very ample curves.
	Hence, it suffices to compute the intersection product with a general smooth projective curve $C\subset S$.
	The preimage $R:=f^{-1}(C)$ is then the blow-up of a minimal ruled surface over $C$ in $C.D$ many points.
	The normal bundle of $R$ in $X$ is given by $\mathcal N_{R/X}=f^\ast \mathcal O_S(C)|_R$.
	Since $T_X|_R=T_R\oplus \mathcal N_{R/X}$, we get
	$$
	(1+c_1(X)|_R+c_2(X)|_R)(1-c_1(\mathcal N_{R/X})+c_1^2(\mathcal N_{R/X}))=1+c_1(R)+c_2(R).
	$$
	Hence, using $\mathcal O_X(R)=f^\ast \mathcal O_S(C)$, we get
	\begin{align*}
	f_\ast c_2(X).C&=c_2(X).R\\
	&=c_2(X)|_R\\
	&=c_2(R)-c_1^2(\mathcal N_{R/X})+c_{1}(\mathcal N_{R/X})c_1(X)|_R\\
	&=c_2(R)-f^\ast \mathcal O_S(C)^3+f^\ast \mathcal O_S(C)^2c_1(X) \\
	&=2-4g(C)+2+C.D+C^2\cdot c_1(\CP^1)\\
	&=-2K_S.C-2C^2+C.D+2C^2\\
	&=-2K_S.C+C.D.
	\end{align*}
	By Lemma \ref{lemma:bettiMfs}, $f_\ast c_1^2(X)\equiv -4K_S-D$.
	Using $p_1=c_1^2-2c_2$, we get
	\begin{align*}
	{f_\ast} p_1(X).C&= f_\ast  c_1^2(X).C-2{f_\ast} c_2(X).C\\
	&=-4K_S.C-D.C+4K_S.C-2D.C\\
	&=-3 D.C,
	\end{align*}
	which proves the lemma.
\end{proof}

\begin{proof}[Proof of Theorem \ref{thm:MFS:surfaces}]
By \cite[Theorem 1.1]{Lin}, any standard conic bundle $f: X \to S$, where $X$ is a smooth K\"ahler threefold, has an algebraic deformation. 
To bound $K_{X_i}^3$ it thus suffices to assume that $X_i$ and $S_i$ are projective for any $i$.

By assumptions, there is an isomorphism $H^2_{tf}(X_i,\Z)\cong H^2_{tf}(X_0,\Z)$ which respects the trilinear forms given by cup products.
We use this isomorphism to identify degree two cohomology classes of $X_i$ with those of $X_0$ (up to torsion).
Using Poincar\'e duality, we further identify classes of $H^4_{tf}(X_i,\Z)$ with linear forms on $H^2_{tf}(X_i,\Z)\cong H^2_{tf}(X_0,\Z)$. 

The codimension one linear subspace ${f_i}^*\PP(H^2(S_i,\Q))$ of $\CP(H^2(X_0,\Q))$ is contained in the cubic hypersurface $\{\alpha\mid \alpha^3=0\}$. 
Passing to a suitable subsequence we can therefore assume that 
$$
f_i^*H^2(S_i,\Q)\subset  H^2(X_0,\Q)
$$ 
does not depend on $i$. 
Let $\ell_i \in H^4_{tf}(X_0,\Z)$ be the class of a fibre of $f_i$.
The linear form determined by this class on $H^2(X_0,\Q)$ has kernel $f_i^*H^2(S_i,\Q)$, and so $\ell_i\cdot \Q$ is independent of $i$.
Since $\ell_i$ is an integral class with $K_{X_i}.\ell_i=-2$, we may after possibly passing to another subsequence assume that $\ell_i=\ell$ does not depend on $i$.  

Since the natural cup product pairing on $H^2(S_i,\Q)$ can be recovered from the pairing
$$
f_i^\ast H^2(S_i,\Q)\times f_i^\ast H^2(S_i,\Q)\longrightarrow \ell\Q,
$$
we get that the pairing on $H^2(S_i,\Q)$ is determined by the cubic form on $H^2(X_0,\Q)$ and so it does not depend on $i$.

Since $f_i^*H^2(S_i,\Q)\subset  H^2(X_0,\Q)$ does not depend on $i$, the same holds for  the homomorphism
$$
\psi_i:H^2(S_i,\Q)\longrightarrow \Q,\ \ \alpha \mapsto p_1(X_i).f_i^\ast \alpha 
$$
By the projection formula, we have $p_1(X_i).f_i^\ast \alpha=(f_i)_\ast p_1(X_i).\alpha$.
Lemma \ref{lem:f_*p1} thus yields
$$
\psi_i(\alpha)=-3D_i.\alpha,
$$
where $D_i$ is the discriminant curve of $f_i$.
This shows that the linear form determined by $[D_i]\in H^2(S_i,\Q)$ on $H^2(S_i,\Q)$ does not depend on $i$.
Since the natural pairing $H^2(S_i,\Q)\times H^2(S_i,\Q)\to \Q$ is perfect by Poincar\'e duality, we get that the class $[D_i]\in H^2(S_i,\Q)$ does not depend on $i$.
Using again the fact that we know the pairing on $H^2(S_i,\Q)$, we finally get that the self-intersection $D_i^2$ does not depend on $i$.

For any class $y\in H^2(X_0,\Q)$, which does not lie in ${f_i}^*H^2(S_i,\Q)$, we have
$$
H^2(X_0, \Q)={f_i}^*H^2(S_i,\Q) \oplus y\cdot \Q \ \ \text{and}\ \ H^4(X_0, \Q)={f_i}^*H^2(S_i,\Q)\cdot y \oplus \ell\cdot \Q .
$$
In particular, $y^2=u y + \lambda \ell$ for some $\lambda \in \Q$ and $u \in {f_i}^*H^2(S_i,\Q)$.
Replacing $y$ by a suitable multiple of $y-\frac{1}{2}u$, we may thus assume that 
$$
y.\ell=-2\ \ \text{and}\ \ y^2 \in {f_i}^*H^4(S_i,\Q)=\ell\cdot \Q .
$$

For any $X_i$, we then get
$$
K_{X_i}=y+f_i^\ast z_i
$$
for some $z_i\in H^2(S_i,\Z)$.
Since
$$
K_{X_i}^3=y^3-6z_i^2,
$$
it suffices to prove the boundedness of $z_i^2$.

Since $y\cdot \ell=-2$, the pushforward of $2yf_i^\ast z_i$ via $f_i$ yields $-4z_i$.
Lemma \ref{lemma:bettiMfs} therefore implies that
$$
(f_i)_\ast K_{X_i}^2\equiv -4z_i\equiv -4 K_{S_i}-D_i.
$$
Hence,
$$
16z_i^2=16 K_{S_i}^2+8 K_{S_i}D_i+ D_i^2.
$$
Since $D_i^2$ does not depend on $i$ and $K_{S_i}^2$ is bounded in terms of the Betti numbers of $S_i$, the statement follows from the fact that $p_a(D_i)$ is bounded by Lemma \ref{lemma:bettiMfs}.
\end{proof}

\section{Uniruled threefolds}

Before we turn to the proof of Theorem \ref{thm:uniruled}, we state few preliminary facts about terminal $\Q$-factorial threefolds.

\subsection{Invariant triples}

Let $X$ be a terminal $\Q$-factorial threefold. 

There exists a well-defined class $c_2(X) \in H^2(X, \Z)^{\vee}=\mathrm{Hom}(H^2(Z,\Z),\Z)$ obtained in the following way (see page 411 in \cite{Reid87}). 
For any $\alpha \in H^2(X,\Z)$ set
$$
c_2(X).\alpha= c_2(\tilde X). f^*\alpha,
$$
where $f: \tilde{X} \to X$ is a resolution of $X$.   

We then define the Pontryagin class $p_1(X) \in  H^2(X, \Q)^{\vee}$ in terms of $c_1(X)$ and $c_2(X)$ in the same way as in the smooth case, where $c_1(X)$ is the class of  $-K_X$ in $H^2(X,\Q)$:

$$
p_1(X):= c_1(X)^2 -2c_2(X).
$$

We also associate to $X$ its cubic form $F_X \in S^3 H^2(X,\Z)^{\vee}$, which is induced by the cup product on $H^2(X,\Z)$.
In this way we can associate to $X$ the triple $(H_{t.f.}^2(X,\Z), F_X, p_1(X))$.   
When $X$ is smooth, this triple encodes many geometrical properties of the 6-manifold underlying $X$ (see for instance \cite{OV95} and \cite{BCT}).

\begin{definition}
	We call $(H_{t.f.}^2(X,\Z), F_X, p_1(X))$ the \emph{invariant triple} of $X$.  
	Two  triples $(H,F,p)$ and $(H',F',p')$, where $H$ (resp.\ $H'$) is a free abelian group,  $F \in S^3H^{\vee}$  (resp.\  $F \in S^3H'^{\vee}$) is a cubic form and $p$ is a linear form on $H\otimes \Q$ (resp.\ $p'$ is a linear form on  $H'\otimes \Q$) are isomorphic if there exists a linear isomorphism $T: H \to H'$ which identifies $F$ with $F'$ and its $\Q$-extension identifies $p$ with $p'$.
\end{definition}

\subsection{Terminal singularities}
We now recall few known facts about terminal singularities in dimension three.

Let $(X,p)$ be the germ of a three-dimensional terminal singularity. 
The {\em index} of $p$ is the smallest positive integer $r$ such that $rK_X$ is Cartier. 
It follows from the classification of terminal singularities, that there exists a deformation of $(X,p)$ into a space with $h\ge 1$ terminal singularities $p_1,\dots,p_h$ which are isolated cyclic quotient singularities of index $r(p_i)$  (for details see \cite[Remark 6.4]{Reid87}). 
The set $\{p_1,\dots,p_h\}$ is called the {\em basket} $\mathcal B(X,p)$ of singularities of $X$ at $p$. 
As in \cite[Section 2.1]{ChenHacon11}, we define 
$$
\Xi(X,p)=\sum_{i=1}^h  r(p_i).
$$
Thus, if $X$ is a projective variety of dimension $3$ with terminal singularities and $\Sing X$ denotes the finite set of singular points of $X$, we may define
$$
\Xi(X)=\sum_{p\in\Sing X} \Xi(X,p).
$$

\subsection{Proof of Theorem \ref{thm:uniruled}}

The following result is interesting by itself and leads naturally to the problem of understanding what kind of topological invariants are determined up to finite ambiguity during a running of an MMP, see Question \ref{q:topMMP}. 
\begin{proposition}{[Cf.\ \cite[Theorem 1.3(2)]{cascini-tasin}]}\label{prop:triple}
	Let $H$ be a finitely generated free abelian group of rank $n+1$, $F \in S^3H^{\vee}$ be a cubic form such that $\Delta_F \ne 0$ and $p$ a linear form on $H$.  Consider the set $\mathcal{P}$ of invariant triples $(H',F',p')$ up to isomorphism, such that there exist
	\begin{enumerate}
		\item a terminal $\Q$-factorial threefold $X$ with associated triple $(H,F,p)$;
		\item a terminal $\Q$-factorial threefold $Y$ with associated triple $(H',F',p')$; 
		\item  a birational morphism $f\colon X\to Y$ which is a divisorial contraction to a point or to a  smooth curve contained  in the smooth locus of $Y$. 
	\end{enumerate}
	Then the set  $\mathcal{P}$ is finite. 
\end{proposition}

\begin{proof}

	Note that the proof of this case works also for $\Delta_F=0$. 
	 Consider the set $\mathcal A$ of primitive elements $\alpha \in H$ such that $\alpha$ is proportional to the exceptional divisor $E$ of some divisorial contraction to a point $f:X \to Y$ as in the statement. 
	 The elements of $\mathcal A$ are points of rank 1 for the Hessian of the cubic form $F$ and so they are finite by \cite[Proposition 3.3]{cascini-tasin}.
	It follows from \cite[Proposition 4.7]{cascini-tasin} that for any sub-module $H'=f^*H^2_{t.f.}(Y,\Z) \hookrightarrow H$ there is $\alpha \in \mathcal A$ such that $\alpha^2.H'=0$ and such that the index of $H' \oplus \Z \alpha$ in $H$ is at most $r^n$, where $r=|\alpha^3|$.
	This implies that for all possible contractions to points $f\colon X \to Y$ as in the statement, the inclusion $f^*H^2_{t.f.}(Y,\Z) \hookrightarrow H^2_{t.f.}(X,\Z)$ is determined up to finite ambiguity.  This determines also $F'$ up to finite ambiguity just restricting $F$ to $H'$.
	
	To prove the finiteness of $p'$ consider a divisorial contraction to a point $f: X \to Y$ and write
	$$
	c_1(X)=f^*c_1(Y) -cE,
	$$
	where $c$ is the discrepancy of the exceptional divisor $E$. 
	Since $c_2(X)=f^*c_2(Y)$ we have that
	$$
	p_1(X)=f^*p_1(Y) -2cf^*c_1(Y).E + c^2E^2 
	$$
	and so  $p_1(Y)$ is given by the restriction of $p_1(X)$ to $f^*H^2_{t.f.}(Y,\Q)$. This means that also  $p'$ is determined up to finite ambiguity and we are done. 
	
	\smallskip
	
	We now look at divisorial contractions to curves. 
	Consider $\mathcal E$ the set of pairs $(E,H')$ where $E$ is a primitive element in $H$ and $H' \subset H$ is a submodule such that
	$$
	H=\Z[E] \oplus H'
	$$
	and the cubic $F$ assumes the form 
	\begin{align} \label{reduced form}
		F=ax_0^3 + \sum_{i=1}^n b_ix_0^2x_i + F'(x_1,\ldots,x_n) 
	\end{align}
	with respect to any basis $E, \alpha_1, \ldots, \alpha_n$ with $\alpha_1, \ldots, \alpha_n \in H'$.
	
	By \cite[Thm. 3.1]{cascini-tasin} there are only finitely many possible non-equivalent reduced forms for $F$. 
	In particular, up to finite ambiguity, we can assume that the coefficients of $F$ in the expression (\ref{reduced form}) are fixed.
	Since the isotropy group of a cubic with non-zero discriminant is finite (\cite[Thm. 4]{OV95}), we deduce that $\mathcal E$ is finite.   
	
	If $f\colon X\to Y$ is a divisorial contraction which contracts a divisor $E$ to a smooth curve $C$ in the smooth locus of $Y$, then (see \cite[Proposition 14]{OV95} and \cite[Proposition 4.8]{cascini-tasin})
	$$
	H^2(X,\Z)= \Z[E] \oplus f^*H^2(Y,\Z)
	$$
	and
	$$
	p_1(X)=f^*(p_1(Y))  + E^2 -2f^*(C).
	$$
	
	Recalling that $E^2.f^*(\alpha)=-C.\alpha$ for any $\alpha \in H^2(Y,\Z)$ we deduce that $p_1(Y)$ is determined by
	 $p_1(X)$,  $E^2$ and by the inclusion $f^*H(Y,\Z) \hookrightarrow H^2(X,\Z)$ and we conclude using the finiteness of $\mathcal E$.  
\end{proof}

\begin{proposition} \label{prop:MFS:singular}
	Let $(X_i)_{i\geq 0}$ be a sequence of terminal $\Q$-factorial threefolds admitting a conic bundle structure $f_i:X_i\longrightarrow S_i$ of relative Picard number 1 over a surface $S_i$. 
	Assume that 
	\begin{enumerate}
	\item the Euler characteristics $\chi(X_i,\mathcal O_{X_i})$ are bounded and $b_2(X_i)=2$;
		\item the invariant triples of $X_0$ and $X_i$ are isomorphic for any $i$;
		\item the sequence $\Xi(X_i)$ is bounded.
	\end{enumerate}
	Then the sequence of Chern numbers $c_1^3(X_i)$ is bounded. 
\end{proposition}

\begin{proof}
	Let $h$ be an ample generator of $\Pic(S_0)$ and let $x\in H^2(X_0,\Z)$ be a primitive class proportional to $f_0^*h$. Then $x^3=0$, $x^2 \ne 0$ and letting $y=c_1(X_0)$ we can write
	$$
	H^2(X_0,\Q)=x\cdot \Q\oplus y\cdot \Q.
	$$
	
	From now on we will use the isomorphism $H^2(X_i,\Q)\cong H^2(X_0,\Q)$ to think about $x$ and $y$ as basis elements of $H^2(X_i,\Q)$. 
	Note that $x^2y \ne 0$ since already $x^3=0$ and $x^2 \ne 0$. 
	Moreover, the space of elements in $H^2(X_0,\C)$ with zero cube is a union of three lines (through $0$) and so we may assume without loss of generality that for each $i$, the pullback of the generator of $H^2(S_i,\Z)$ to $X_i$ is a multiple of $x$. 
	In particular, $x^2$ is a multiple of the class of the general fibre of $X_i\longrightarrow S_i$ for all $i$.
	
	We have
	$$
	c_1(X_i) = a_i\cdot x+ b_i\cdot y, 
	$$
	for some $a_i,b_i \in \Q$.  
	Since $\Xi(X_i)$ is bounded, there is a positive integer $r$ such that $r K_{X_i}$ is Cartier for any $i$. In particular, $ra_i, rb_i \in \Z$. 
	Since $K_{X_i}.C=-2$ where $C$ is a general fibre, we deduce that the sequence of $b_i$ is bounded.

	We are going to bound the sequence of $a_i$.
	By the singular version of Riemann--Roch \cite[Corollary 10.3]{Reid87} we get
	$$
	48 \chi(X_i,\mathcal O_{X_i})= c_1(X_i).p_1(X_i) - c_1(X_i)^3 + T_i
	$$
	where
	$$
	T_i=\sum_{p_\alpha} \left (r(p_\alpha)-\frac 1 {r(p_\alpha)}\right ),
	$$
	and the sum runs over all the points of all the baskets of $X_i$. Note that $T_i$ is a bounded sequence since $\Xi(X_i)$ is bounded. This implies that
	$$
	48 \chi(X_i,\mathcal O_{X_i})= -3a_i^2b_ix^2y + a_ix( 2p_1(X_i) -3b_i^2y^2) +b_i^3y^3 +b_iyp_1(X_i)+T_i
	$$
	and so the $a_i$ are also bounded, since  $b_ix^2y \ne 0$ and $\chi(X_i,\mathcal O_{X_i})$ are bounded.	
\end{proof}

\begin{proof}[Proof of Theorem \ref{thm:uniruled}]
	Let $f:X\to Y$ be the birational contraction as in Theorem \ref{thm:uniruled}.
	By the proof of \cite[Corollary 1.5]{cascini-tasin}, we know that $|K_X^3 - K_Y^3|$ is bounded by a constant depending only on the Betti numbers of $X$ and on the cubic form $F_X$. To conclude we need to bound $K_Y^3$ in terms of the topology of $X$.

	Since $Y$ is a Mori fibre space and $\Delta_{F_Y} \ne 0$, we deduce that either $Y$ is a Fano variety or $Y$ has a conic bundle structure over a surface with second Betti number 1 (otherwise there would be an element in $H^2(X,\C)$ with square zero, which would imply that $\{F=0\}$ has a singular point and so $\Delta_{F_Y} = 0$). 
	Since terminal Fano threefolds are bounded, we are left with the conic bundle case.
	
	Proposition \ref{prop:triple} assures us that the invariant triple of $Y$ is determined up to finite ambiguity by the invariant triple of $X$. 
	Moreover, the Euler characteristic $\chi(Y,\mathcal O_Y)=\chi(X,\mathcal O_X)$ is bounded in terms of the Betti numbers of $X$ and by \cite[Prop. 3.3]{CZ14} we also have a bound for  $\Xi(Y)$ depending only on $b_2(X)$. 
	The result follows then from Proposition \ref{prop:MFS:singular}.
\end{proof}

\section*{Acknowledgments}
We thank P.\ Cascini for many useful conversations on the topic of this manuscript. 
The first author is supported by DFG grant ``Topologische Eigenschaften von Algebraischen Variet\"aten'' (project nr.\ 416054549).
The second author was supported by the DFG grant ``Birational Methods in Topology and Hyperk\"ahler Geometry" and is a member of the GNSAGA group of INdAM.

\end{document}